\documentclass[a4paper,12pt,reqno]{amsart}
\usepackage{amsmath, amssymb,amsthm}
\usepackage{enumerate}
\usepackage{cite}
\usepackage{graphicx}
\usepackage{subcaption}
\numberwithin{equation}{section}
\newtheorem{theorem}{Theorem}[section]
\newtheorem{lemma}[theorem]{Lemma}

\newtheorem{corollary}[theorem]{Corollary}
\newtheorem{remark}[theorem]{Remark}

\newtheorem{giep}{GIEP}
\theoremstyle{definition}

\begin{document}
\title[A generalized inverse eigenvalue problem and $m$-functions]
{A generalized inverse eigenvalue problem and $m$-functions}
\author[K. K. Behera]{Kiran Kumar Behera}
\address{
Department of Mathematics,
Indian Institute of Science, Bangalore, India}
\email{kiranbehera@iisc.ac.in}
\thanks{
This research is supported by the Dr. D. S. Kothari postdoctoral
fellowship scheme of University Grants Commission (UGC), India.
}
%The author gratefully
%acknowledges his mentor Prof. Gadadhar Misra.}
%
 \subjclass[2010]{Primary 15A29, 30C10}
\keywords
{
Generalized inverse eigenvalue problem;
Linear pencil of tridiagonal matrices;
$m$-functions
}
\date{}
\dedicatory{}
\commby{}
\begin{abstract}
In this manuscript, a generalized inverse eigenvalue problem is considered
that involves a linear pencil
$(z\mathcal{J}_{[0,n]}-\mathcal{H}_{[0,n]})$
of matrices arising in the theory of rational interpolation
and biorthogonal rational functions.
In addition to the reconstruction of the Hermitian matrix
$\mathcal{H}_{[0,n]}$ with the entries $b_j's$,
characterizations of the rational functions that are components of the
prescribed eigenvectors are given.
A condition concerning the positive-definiteness of
$\mathcal{J}_{[0,n]}$ and which is often an
assumption in the direct problem is also isolated.
Further, the reconstruction of $\mathcal{H}_{[0,n]}$
is viewed through the inverse of the pencil
$(z\mathcal{J}_{[0,n]}-\mathcal{H}_{[0,n]})$
which involves the concept of $m$-functions.
\end{abstract}
\maketitle
\section{Introduction}
\label{sec: Introduction}
A generalized inverse eigenvalue problem (GIEP) 
concerns the reconstruction of matrices from a
given set of spectral data. The spectral data may be completely or only
partially specified in terms of eigenvalues and eigenvectors.
Precisely, a GIEP for a pair $(\mathcal{H},\mathcal{J})$ 
of square matrices involves the
generalized eigenvalue equation 
$\mathcal{H}\Phi=z\mathcal{J}\Phi$.
With the prescribed spectral data, the solution to the 
problem consists in the reconstruction
of the matrices $\mathcal{H}$ and/or $\mathcal{J}$
\cite{Lancaster-Ye-IGEP-LAA-1988,
Dai-Yuan-GIEP-JCAM-2009}.

In general, it is often necessary both from the point of view of
practical applications and of mathematical interest that the matrices
involved have a specified structure
\cite{Chu-SIAM-review-1998}.
This introduces a structural constraint on the solution
in addition to the spectral constraint. 
Thus, one may require that both the matrices
$\mathcal{H}$ and $\mathcal{J}$ or one of them to be, for instance, 
banded or Hermitian or Hamiltonian 
\cite{GIEP-Hamiltonian-AMC-2019,
Sen-Sharma-LAA-2014}
and so on.

In the present manuscript,
we consider, as an inverse problem, the generalized 
eigenvalue equation arising
from the continued fraction
\begin{align}
\label{eqn: continued fraction for eigenvalue problem}
\dfrac{1}{u_0(z)-\dfrac{v_0^{L}(z)v_0^{R}(z)}{u_1(z)-
\dfrac{v_1^{L}(z)v_2^{R}(z)}{u_2(z)-\ddots}}},
\end{align}
where $u_j(z)$, $v_j^{L}(z)$ and $v_j^{R}(z)$
are non-vanishing polynomials of degree one
\cite{Baker-Morris-Pade-1981}.
If we terminate the above continued fraction at $u_n(z)$,
then it is a rational function denoted by
$\mathcal{S}_{n+1}(z)=\mathcal{Q}_{n+1}(z)/\mathcal{P}_{n+1}(z)$
where, the polynomials $\mathcal{Q}_n(z)$ of degree $\leq n-1$
and $\mathcal{P}_n(z)$ of degree $\leq n$
satisfy the three term recurrence relation
\cite{Wall-book-1948}
\begin{align}
\label{eqn: three term recurrence relation for m-function}
\mathcal{X}_{n+1}(z)=u_n(z)\mathcal{X}_n(z)-
v_{n-1}^{L}(z)v_{n-1}^{R}(z)\mathcal{X}_{n-1}(z),
\quad
n=0,1,2,\cdots,
\end{align}
with the initial conditions
\begin{align}
\label{eqn: initial conditions}
\mathcal{Q}_{-1}(z)=-1,
\quad\mathcal{Q}_0(z)=0,
\quad\mathcal{P}_{-1}(z)=0
\quad\mbox{and}\quad
\mathcal{P}_0(z)=1.
\end{align}
The linear pencil that is associated with
$\mathcal{S}_{n+1}(z)$ and
\eqref{eqn: three term recurrence relation for m-function}
is $(z\mathcal{J}_{[0,n]}-\mathcal{H}_{[0,n]})$
where the matrices
$\mathcal{H}_{[0,n]}$ and $\mathcal{J}_{[0,n]}$
are tridiagonal.

In the present manuscript, we consider the matrices
\begin{align*}
\mathcal{H}_{[0,n]}=
\left(
  \begin{array}{cccccc}
    a_0 & b_0 & 0 & \cdots & 0 & 0 \\
    \bar{b}_0 & a_1 & b_1 & \cdots & 0 & 0 \\
    0 & \bar{b}_1 & a_2 & \cdots & 0 & 0 \\
    \vdots & \vdots & \vdots & \ddots & \vdots & \vdots \\
    0 & 0 & 0 & \cdots & a_{n-1} & b_{n-1} \\
    0 & 0 & 0 & \cdots & \bar{b}_{n-1} & a_n \\
  \end{array}
\right),
\end{align*}
\begin{align*}
\mathcal{J}_{[0,n]}=
\left(
  \begin{array}{cccccc}
    c_0 & d_0 & 0 & \cdots & 0 & 0 \\
    d_0 & c_1 & d_1 & \cdots & 0 & 0 \\
    0 & d_1 & c_2 & \cdots & 0 & 0 \\
    \vdots & \vdots & \vdots & \ddots & \vdots & \vdots \\
    0 & 0 & 0 & \cdots & c_{n-1} & d_{n-1} \\
    0 & 0 & 0 & \cdots & d_{n-1} & c_n \\
  \end{array}
\right),
\quad
d_j\neq0,
\quad
0\leq j\leq n-1,
\end{align*}
and the following generalized inverse eigenvalue problem.
\begin{giep}
\label{giep: GIEP}
Given:
the symmetric matrix $\mathcal{J}_{[0,n]}$,
the hermitian matrix $\mathcal{H}_{[0,k]}$,
real numbers $\lambda$ and $\mu$,
and vectors
$p_{[k,n]}^{R}=(p_{k}^{R},p_{k+1}^{R},\cdots,p_n^{R})^{T}$
and
$s_{[k,n]}^{R}=(s_{k}^{R},s_{k+1}^{R},\cdots,s_n^{R})^{T}$,
where $1\leq k\leq n-1$.
To find:
\begin{enumerate}[(i)]
\item
hermitian matrix $\mathcal{H}_{[0,n]}$ with
eigenvalues $\lambda$ and $\mu$
such that $\mathcal{H}_{[0,k]}$ is the leading principal
sub-matrix of $\mathcal{H}_{[0,n]}$,
\item
vectors $p_{[0,k-1]}^{R}=(p_0^{R},p_1^{R},\cdots,p_{k-1}^{R})^{T}$
and $s_{[0,k-1]}^{R}=(s_0^{R},s_1^{R},\cdots,s_{k-1}^{R})^{T}$
such that
\begin{align*}
p_{[0,n]}^{R}=
\left(
  \begin{array}{c}
    p_{[0,k-1]}^{R} \\
    p_{[k,n]}^{R} \\
  \end{array}
\right)
\quad\mbox{and}\quad
s_{[0,n]}^{R}=
\left(
  \begin{array}{c}
    s_{[0,k-1]}^{R} \\
    s_{[k,n]}^{R} \\
  \end{array}
\right),
\end{align*}
are the right eigenvectors of the matrix pencil
$(z\mathcal{J}_{[0,n]}-\mathcal{H}_{[0,n]})$,
corresponding to the
eigenvalues $\lambda$ and $\mu$ respectively.
\end{enumerate}
\end{giep}
The pencil $z\mathcal{J}_{[0,n]}-\mathcal{H}_{[0,n]}$,
which is a linear pencil of tridiagonal matrices
arises in the theory of biorthogonal rational functions
and rational interpolation
\cite{Beckermann-D-Z-linear-pencil-JAT-2010,
Zhedanov-GEP-JAT-1998}.
A particular case, in which the $b's$ appearing in
$\mathcal{H}_{[0,n]}$
are purely imaginary and the $c's$ appearing in
$\mathcal{J}_{[0,n]}$
 are unity, has its origins in the continued fraction
representation of Nevanlinna functions, which in turn
are obtained via the Cayley transformation
of the continued fraction representation of a
Carath\'{e}odory function
\cite{Wall-book-1948}
(see also \cite{Maxim-note-JAT-2017}).
As further specific illustrations, the rational functions
arising as components of eigenvectors
in such cases have been related to a class of
hypergeometric polynomials orthogonal on the unit
circle
\cite{Ismail-Ranga-GEP-LAA-2019}
as well as pseudo-Jacobi polynomials
(or Routh-Romanovski polynomials)
\cite{Maxim-note-JAT-2017}.

In the direct problem, the components of the right eigenvector
of the linear pencil
$(z\mathcal{J}_{[0,n]}-\mathcal{H}_{[0,n]})$
are rational functions with poles at $z=b_j/d_j$ while that of the
left eigenvector have poles at $z=\bar{b}_j/d_j$.
In addition to poles, the entries of the matrices
$\mathcal{J}_{[0,n]}$ and $\mathcal{H}_{[0,n]}$
also completely specify the numerators of such rational functions
appearing as polynomial solutions of a three term recurrence relation.

It is also known that the zeros of these numerator polynomials are the
eigenvalues of the linear pencil under consideration
\cite[Theorem 1.1]{Ismail-Ranga-GEP-LAA-2019}.
These numerator polynomials are precisely normalized
$\mathcal{P}_{j}(z)$, the denominator of the
convergents $\mathcal{S}_{j}(z)$ of the continued fraction
\eqref{eqn: continued fraction for eigenvalue problem}.
Further, it can be verified with the recurrence relation
\eqref{eqn: three term recurrence relation for m-function}
and the initial conditions
\eqref{eqn: initial conditions}
that the following expressions
\begin{align*}
\mathcal{P}_{n+1}(z)=\det{(z\mathcal{J}_{[0,n]}-\mathcal{H}_{[0,n]})},
\quad
\mathcal{Q}_{n+1}(z)=\det{(z\mathcal{J}_{[1,n]}-\mathcal{H}_{[1,n]})},
\end{align*}
hold which lead to the formula
\begin{align*}
\mathcal{S}_{n+1}(z)=\frac{\mathcal{Q}_{n+1}(z)}{\mathcal{P}_{n+1}(z)}=
\left\langle
(z\mathcal{J}_{[0,n]}-\mathcal{H}_{[0,n]})^{-1}e_0,e_0
\right\rangle,
\end{align*}
with the standard inner product 
\begin{align*}
\langle
x,y
\rangle
=
\sum_{j=0}^{\infty}x_i\bar{y}_j,
\quad
x=(x_0,x_1,\cdots)\in\ell^2,
\quad
y=(y_0,y_1,\cdots)\in\ell^2.
\end{align*}
on the space
$\ell^{2}$ of complex square summable sequences.

A fundamental object related to a pair $(\mathcal{H}, \mathcal{J})$ of matrices
is the function
\begin{align*}
\mathfrak{m}(z)=
\left\langle
(z\mathcal{J}-\mathcal{H})^{-1}e_0,e_0
\right\rangle,
\quad
z\in\rho(\mathcal{H},\mathcal{J})
\end{align*}
called the $m$-function or the Weyl function of the linear pencil
$(z\mathcal{J}-\mathcal{H})$
\cite{Beckermann-D-Z-linear-pencil-JAT-2010}
(see also
\cite{Aptekarev-resolvent-criteria-PAMS-1995,
Beckermann-Weyl-function-Const.-Apprx-1997}).
Here $\sigma(\mathcal{H},\mathcal{J})$
and $\rho(\mathcal{H},\mathcal{J}):=
\mathbb{C}\setminus\sigma(\mathcal{H},\mathcal{J})$
are, respectively, the spectrum and the resolvent set
of the pencil $(z\mathcal{J}-\mathcal{H})$.
We can denote similarly the $m$-function
\begin{align}
\label{eqn: m-functions for finite pencil}
\mathfrak{m}(z,j+1)=
\frac{\mathcal{Q}_{j+1}(z)}{\mathcal{P}_{j+1}(z)}=
\left\langle
(z\mathcal{J}_{[0,j]}-\mathcal{H}_{[0,j]})^{-1}e_0,e_0
\right\rangle,
\quad
z\in\rho(\mathcal{H}_{[0,j]},\mathcal{J}_{[0,j]}),
\end{align}
of the finite pencil
$(z\mathcal{J}_{[0,j]}-\mathcal{H}_{[0,j]})$.

Thus, a way to interpret the reconstruction of the matrix
$\mathcal{H}_{[0,n]}$
is to determine its entries in terms of rational functions
with arbitrary poles. These rational functions enter into
the problem as components of a prescribed eigenvector,
while the structural constraint of the pencil being tridiagonal
characterizes these poles.

Our primary goal in this manuscript is to find a representation of the entries
$b_j's$ of the matrix $\mathcal{H}_{[0,n]}$ in terms of given spectral points
and corresponding eigenvectors.
We find characterizations of both the given poles and the entries $b_j's$
which appear in special matrix pencils as mentioned earlier.
A condition concerning the positive-definiteness of
$\mathcal{J}_{[0,n]}$ and which is often an
assumption in the direct problem is also isolated.
Further, we have a view at the entries $b_j's$ through
the $\mathfrak{m}$-functions
\eqref{eqn: m-functions for finite pencil}
which, as is obvious, involve a
point in the resolvent set and not in the spectrum of the pair
$(\mathcal{H}_{[0,n]},\mathcal{J}_{[0,n]})$.

The manuscript is organized as follows.
Section~\ref{sec: preliminary results}
includes preliminary results that illustrate the key role played by the
entry $b_k$ in the inverse approach to the linear pencil
$(z\mathcal{J}_{[0,n]}-\mathcal{H}_{[0,n]})$.
The matrix $\mathcal{H}_{[0,n]}$ is reconstructed
in Section~\ref{sec: solution to the giep}
thereby solving GIEP~\ref{giep: GIEP}.
In Section~\ref{sec: view with m-functions}
we have a further look at the problem through
$m$-functions that involves computing the inverse of the matrix
$(z\mathcal{J}_{[0,n]}-\mathcal{H}_{[0,n]})$.
\section{Preliminary results}
\label{sec: preliminary results}
In this section, we derive some results that will help in solving the GIEP.
Though the entries are yet to be determined, we use them as
symbols in the computation, with the final result depending only on $b_k$
and given components of the eigenvector. In a way, these results exhibit the
role played by the specific entry $b_k$ in the solution.

First of all, it can be seen that
if $\mathcal{H}_{[0,n]}$ were completely specified,
the leading minors $\mathcal{P}_{m}(z)$
of $(z\mathcal{J}_{[0,n]}-\mathcal{H}_{[0,n]})$
satisfy the three term recurrence relation
\begin{align}
\label{eqn: recurrence relation satisfied by numerator polynomials}
\mathcal{X}_{m+1}(z)=(z c_m-a_m)\mathcal{X}_m(z)-
(z d_{m-1}-b_{m-1})(z d_{m-1}-\bar{b}_{m-1})
\mathcal{X}_{m-1}(z),
\end{align}
for $0\leq m\leq n$,
where we define $\mathcal{P}_{-1}(z):=0$ and $\mathcal{P}_0(z):=1$.
If $\kappa_m$ is the leading coefficient of $\mathcal{P}_m(z)$, then from
\eqref{eqn: recurrence relation satisfied by numerator polynomials}, we have
$\kappa_{m+1}=c_m\kappa_m-d_{m-1}^2\kappa_{m-1}$, with $\kappa_0=1$
and $\kappa_1=c_0$.
Hence, if
\begin{align*}
\frac{\kappa_m}{\kappa_{m-1}}\neq\frac{d_{m-1}^2}{c_m},
\quad
m\geq1,
\end{align*}
then $\mathcal{P}_{m+1}(z)$ is a polynomial of degree $m+1$.
Further,
$(z\mathcal{J}_{[0,n]}-\mathcal{H}_{[0,n]})p_{[0,n]}^{R}=0$
yields the following relations
\begin{align}
\label{eqn: rational equation for components}
(z d_{m-1}-\bar{b}_{m-1})p_{m-1}^{R}(z)+
(z c_m-a_m)p_m^{R}(z)+
(z d_m-b_m)p_{m+1}^{R}(z)=0,
\end{align}
for $m=0,1,\cdots,n$, where $p_{n+1}^{R}(z)=0$
and we define $p_{-1}^{R}(z):=0$.
The former equality occurs if
$z\in\sigma(\mathcal{H}_{[0,n]}, \mathcal{J}_{[0,n]})$.
Moreover, with $p_0^{R}(z)$ a non-vanishing function to be specified,
the components of the eigenvector $p_{[0,n]}^{R}$ can be obtained from
\eqref{eqn: rational equation for components},
for instance by induction,
in the form of the rational functions
\begin{align}
\label{eqn: components of the eigenvector Phi}
p_m^{R}(z)=\frac{\mathcal{P}_m(z)}
{\prod_{j=0}^{m-1}(b_j-zd_j)}p_0^{R}(z),
\quad
m=1,2,\cdots,k,k+1,\cdots,n-1,
\end{align}
and $p_n^{R}(z)$ obtained from
$(zc_n-a_n)p_n^{R}(z)=
(\bar{b}_{n-1}-z d_{n-1})p_{n-1}^{R}(z)$.
However, because of the prescribed data,
we will assume that the components of the vector
$p_{[k,n]}^{R}$
are given in the form
\begin{align}
\label{eqn: given form of rational functions}
p_m^{R}(z)=\frac{\mathcal{T}_{m}(z)}
{\eta_{m}^{(z)}
\prod_{j=0}^{m-1}(\alpha_j-z)},
\quad
\eta_{m}^{(z)}\in\mathbb{R}\setminus\{0\},
\quad
m=k,k+1,\cdots,n,
\end{align}
where $\mathcal{T}_m(z)$ is a polynomial of degree $m$
with leading coefficient $\varsigma_m$.
But, we note that once $\mathcal{H}_{[0,n]}$
is determined,
a component of the right eigenvector of
$(z\mathcal{J}_{[0,n]}-\mathcal{H}_{[0,n]})$
must be of the form
\eqref{eqn: components of the eigenvector Phi}.
In particular, if we look at $p_k^{R}(z)$,
this would imply that the set
$\{\alpha_0,\alpha_1,\cdots,\alpha_{k-1}\}$
is necessarily a permutation of the set
$\{b_0/d_0, b_1/d_1,\cdots,b_{k-1}/d_{k-1}\}$
which is known.
The inverse problem thus consists of $\alpha_j$,
$j=k,\cdots,n-1$, being arbitrary, 
on which the determination of $b_j$, $j=k,\cdots,n-1$, depends.
\begin{lemma}
\label{lemma: vector Phi-1}
Given $(\lambda, p_{[0,n]}^{R})$ an eigen-pair for
$(\mathcal{H}_{[0,n]}, \mathcal{J}_{[0,n]})$,
let $\lambda\notin\sigma(\mathcal{H}_{[0,k]},
\mathcal{J}_{[0,k]})$,
then the components of
$p_{[0,k-1]}^{R}(z)$ are given by
\begin{align*}
p_m^{R}(z)
&=
\frac{(b_k-z d_k)\prod_{j=m}^{k-1}(b_j-z d_j)\mathcal{P}_m(z)}
{\mathcal{P}_{k+1}(z)}p_{k+1}^{R}(z),
\quad
m=1,2,\cdots, k-1,
\end{align*}
at $z=\lambda$ and $p_0^{R}(\lambda)$ 
assumed to be a non-vanishing function of $\lambda$.
\end{lemma}
\begin{proof}
Since $\lambda\notin \sigma(\mathcal{H}_{[0,k]},\mathcal{J}_{[0,k]})$,
$\det(\lambda\mathcal{J}_{[0,k]}-\mathcal{H}_{[0,k]})\neq0$,
which implies that $\mathcal{P}_{k+1}(\lambda)\neq0$.
If we use the form of $p_{k+1}^{R}(\lambda)$
as suggested by \eqref{eqn: components of the eigenvector Phi}
with $b_k$ unknown at the moment,
we will have that $p_{k+1}^{R}(\lambda)\neq0$ and
\begin{align*}
p_0^{R}(\lambda)=
\frac{\prod_{j=0}^{k}(b_j-\lambda d_j)p_{k+1}^{R}(\lambda)}
{\mathcal{P}_{k+1}(\lambda)}=
\frac{(b_k-\lambda d_k)p_{k+1}^{R}(\lambda)}
{\mathcal{P}_{k+1}(\lambda)}
\prod_{j=0}^{k-1}(b_j-\lambda d_j)\mathcal{P}_0(\lambda).
\end{align*}
Then, $p_1^{R}(\lambda)$ can be obtained using
\eqref{eqn: rational equation for components}
for $m=0$ as
$(\lambda c_0-a_0)p_0^{R}(\lambda)+
(\lambda d_0-b_0)p_1^R(\lambda)=0$
giving
\begin{align*}
p_1^{R}(\lambda)=\frac{(b_k-\lambda d_k)
p_{k+1}^{R}(\lambda)}
{\mathcal{P}_{k+1}(\lambda)}
\prod_{j=1}^{k-1}(b_j-\lambda d_j)
\mathcal{P}_1(\lambda).
\end{align*}
The rest of the proof can be completed by induction
using
\eqref{eqn: rational equation for components}
and
\eqref{eqn: components of the eigenvector Phi}.
\end{proof}
We note that $\mathcal{P}_{k+1}(\lambda)$ is known since
\eqref{eqn: recurrence relation satisfied by numerator polynomials}
involves $a_k$ and $b_{k-1}$.
Thus, Lemma \ref{lemma: vector Phi-1} shows that once $b_k$ is computed,
the vector $p_{[0,k-1]}^{R}$ can be uniquely obtained in terms of
$p_{k+1}^{R}(\lambda)$, which is now known in the form
given by \eqref{eqn: given form of rational functions}.
This would determine $p_{[0,k-1]}^{R}$ at $z=\lambda$
completely.

To proceed further, we will make use of rational functions
of the form
\begin{align}
\label{eqn: components of the left eigenvector Psi}
p_m^{L}(z)=\frac{\mathcal{P}_m(z)}
{\prod_{j=0}^{m-1}(\bar{b}_j-z d_j)}
p_0^{L}(z),
\quad
m=1,2\cdots, k-1,k,\cdots,n-1,
\end{align}
and $p_{n}^{L}(z)$ obtained from the equation
$(zc_n-a_n)p_n^{L}(z)=
(b_{n-1}-z d_{n-1})p_{n-1}^{L}(z)$.
These arise as components of the left eigenvector $p_{[0,n]}^{L}$
of $(z\mathcal{J}_{[0,n]}-\mathcal{H}_{[0,n]})$
corresponding to the eigenvalue $z=\lambda$
and owing to the underlying hermitian character of the
problem, satisfy
$p_j^{L}(\lambda)=\overline{p_j^{R}(\lambda)}$, $0\leq j\leq n$.
Again, we reiterate that $p_j^{L}(z)$ is specified
in the form \eqref{eqn: components of the left eigenvector Psi}
only for
$j=0,1,\cdots, k-1$, while $p_m^{L}(z)$ is obtained in the form
suggested by
\eqref{eqn: given form of rational functions} for $m=k,k+1,\cdots,n$.

Similarly, we define the rational functions $s_m^{L}(\mu)$ and $s_m^{R}(\mu)$
corresponding to the eigenvalue $\mu$. Lemma~\ref{lemma: vector Phi-1}
with $\lambda$ replaced by $\mu$ and the assumption
$\mu\notin\sigma(\mathcal{H}_{[0,k]},\mathcal{J}_{[0,k]})$
gives the corresponding expressions for
$s_j^{R}(z)$ at $z=\mu$ for $j=0,1,\cdots,k-1$.

For ease of notations, we use $p_j^{R}:=p_j^{R}(z)$ and similarly for others.
Let us also denote by $\mathcal{J}_{[k+1,n]}$ and $\mathcal{H}_{[k+1,n]}$,
as is clear from the notations,
the trailing matrices obtained by removing the first $k+1$ rows and columns
from
$\mathcal{J}_{[0,n]}$ and $\mathcal{H}_{[0,n]}$ respectively.
\begin{lemma}
\label{lemma: result from equating trace}
Suppose $\lambda,\mu\notin\sigma(\mathcal{H}_{[0,i]}, \mathcal{J}_{[0,i]})$
for $i=k-1,k$.
Then, the following identities
\begin{align}
\label{eqn: Wronskian-result for trailing J before mu tends to lambda}
(\lambda-\mu)p_{[k+1,n]}^{L}\mathcal{J}_{[k+1,n]}s_{[k+1]}^{R}
&=
(b_k-\lambda d_k)p_k^{L}s_{k+1}^{R}-(\bar{b}_k-\mu d_{k})p_{k+1}^{L}s_{k}^{R},
\\
\label{eqn: Wronskian-result for J-(k+1) before mu tends to lambda}
(\lambda-\mu)s_{[0,k]}^{L}\mathcal{J}_{[0,k]}p_{[0,k]}^{R}
&=
(\bar{b}_k-\mu d_k)s_{k+1}^{L}p_{k}^{R}-(b_k-\lambda d_{k})s_{k}^{L}p_{k+1}^{R},
\end{align}
hold.
\end{lemma}
\begin{proof}
The relations \eqref{eqn: rational equation for components}
for $m=k+1\cdots,n$ can be written as
\begin{align*}
\lefteqn{
z
\left(
  \begin{array}{ccccc}
    c_{k+1} & d_{k+1} & 0 & \cdots & 0 \\
    d_{k+1} & c_{k+2} & d_{k+2} & \cdots & 0 \\
    0 & d_{k+2} & c_{k+3} & \cdots & 0 \\
    \vdots & \vdots & \vdots & \ddots & \vdots \\
    0 & 0 & 0 & \cdots & c_n \\
  \end{array}
\right)
\left(
  \begin{array}{c}
    p_{k+1}^{R} \\
    p_{k+2}^{R} \\
    p_{k+3}^{R} \\
    \vdots \\
    p_{n}^{R} \\
  \end{array}
\right)
}
\\
&&=
\left(
  \begin{array}{ccccc}
    a_{k+1} & b_{k+1} & 0 & \cdots & 0 \\
    \bar{b}_{k+1} & a_{k+2} & d_{k+2} & \cdots & 0 \\
    0 & \bar{b}_{k+2} & a_{k+3} & \cdots & 0 \\
    \vdots & \vdots & \vdots & \ddots & \vdots \\
    0 & 0 & 0 & \cdots & a_n \\
  \end{array}
\right)
\left(
  \begin{array}{c}
    p_{k+1}^{R} \\
    p_{k+2}^{R} \\
    p_{k+3}^{R} \\
    \vdots \\
    p_{n}^{R} \\
  \end{array}
\right)
+
(\bar{b}_k-z d_k)
\left(
  \begin{array}{c}
    p_{k}^{R} \\
    0 \\
    0 \\
    \vdots \\
    0 \\
  \end{array}
\right),
\end{align*}
or in the compact form
\begin{align}
\label{eqn: trailing matrix relation in lemma}
\mathcal{H}_{[k+1,n]}p_{[k+1,n]}^{R}=z\mathcal{J}_{[k+1,n]}p_{[k+1,n]}^{R}-
(\bar{b}_k-z d_k)p_k^{R}\vec{e}_1,
\end{align}
where $\vec{e}_1=(1,0,\cdots,0)\in\mathbb{R}^{n-k}$.
Post-multiplying \eqref{eqn: trailing matrix relation in lemma}
at $z=\lambda$ by $s_{[k+1,n]}^{L}$, we obtain
\begin{align}
\label{eqn: trailing matrix relation in lemma-post}
\mathcal{H}_{[k+1,n]}p_{[k+1,n]}^{R}s_{[k+1,n]}^{L}=
\lambda\mathcal{J}_{[k+1,n]}p_{[k+1,n]}^{R}s_{[k+1,n]}^{L}-
(\bar{b}_k-\lambda d_k)p_k^{R}\vec{e}_1s_{[k+1,n]}^{L},
\end{align}
 while pre-multiplying the conjugate transpose of
\eqref{eqn: trailing matrix relation in lemma}
at $z=\mu$ by $p_{[k+1,n]}^{R}$, we obtain
\begin{align}
\label{eqn: trailing matrix relation in lemma-pre}
p_{[k+1,n]}^{R}s_{[k+1,n]}^{L}\mathcal{H}_{[k+1,n]}=
\mu p_{[k+1,n]}^{R}s_{[k+1,n]}^{L}\mathcal{J}_{n+1}^{(k+1)}-
(b_k-\mu d_k)s_k^{L}p_{[k+1,n]}^{R}\vec{e}^{T}_1.
\end{align}
We proceed with the well-known technique of subtracting
traces of the respective sides of
\eqref{eqn: trailing matrix relation in lemma-post}
and
\eqref{eqn: trailing matrix relation in lemma-pre}.
The left hand side upon subtraction is zero owing to the fact that
$\mathrm{Tr}(A.B)=\mathrm{Tr}(B.A)$ for any well-defined
matrix product.
Consequently, we have
\begin{align*}
(\lambda-\mu)\mathrm{Tr}[p_{[k+1,n]}^{R}s_{[k+1,n]}^{L}\mathcal{J}_{[k+1,n]}]=
(\bar{b}_k-\lambda d_k)p_k^{R}\mathrm{Tr}[\vec{e}_1s_{[k+1,n]}^{L}]-
(b_k-\mu d_k)s_k^{L}\mathrm{Tr}[p_{[k+1,n]}^{R}\vec{e}^{T}_1].
\end{align*}
The left hand side above is equal to the matrix product
$(\lambda-\mu)s_{[k+1,n]}^{L}\mathcal{J}_{[k+1,n]}p_{[k+1,n]}^R$,
while the right hand side can be simplified to
$(\bar{b}_k-\lambda d_k)s_{k+1}^{L}p_k^{R}-
(b_k-\mu d_k)s_k^{L}p_{k+1}^R$
which gives
\eqref{eqn: Wronskian-result for trailing J before mu tends to lambda}.
A similar computation starting from
\begin{align*}
\mathcal{H}_{[0,k]}p_{[0,k]}^{R}=z\mathcal{J}_{[0,k]}p_{[0,k]}^{R}-
(b_k-z d_k)p_{k+1}^{R}\vec{e}_{k+1},
\quad
\vec{e}_{k+1}=(0,0\cdots,1)\in\mathbb{R}^{k+1},
\end{align*}
leads to
\eqref{eqn: Wronskian-result for J-(k+1) before mu tends to lambda}.
\end{proof}
The assumptions in Lemma~\ref{lemma: result from equating trace} are necessary for 
the right hand sides of 
\eqref{eqn: Wronskian-result for trailing J before mu tends to lambda}
and
\eqref{eqn: Wronskian-result for J-(k+1) before mu tends to lambda}
to be non-vanishing.
Later, we will use \eqref{eqn: Wronskian-result for J-(k+1) before mu tends to lambda}
to make an observation regarding the positive-definiteness of $\mathcal{J}_{[0,k]}$.
Before that we solve the stated inverse problem
GIEP~\ref{giep: GIEP}.
\section{Solution to the GIEP}
\label{sec: solution to the giep}
The given data in GIEP~\ref{giep: GIEP} suggest
that we write the equation
$(\lambda\mathcal{J}_{[0,n]}-\mathcal{H}_{[0,n]})p_{[0,n]}^{R}=0$
in the form
\begin{align}
\label{eqn: eigenvalue relation suggested by data}
\left(
  \begin{array}{cc}
    \lambda\mathcal{J}_{[0,k]}-\mathcal{H}_{[0,k]} & \mathcal{O}_{\lambda} \\
    \mathcal{O}_{\lambda}^{\ast} & \lambda\mathcal{J}_{[k+1,n]}-\mathcal{H}_{[k+1,n]} \\
  \end{array}
\right)
\left(
  \begin{array}{c}
    p_{[0,k]}^{R} \\
    p_{[k+1,n]}^{R} \\
  \end{array}
\right)
=
\left(
  \begin{array}{c}
    0 \\
    0 \\
  \end{array}
\right),
\end{align}
where
\begin{align}
\label{eqn: O-lambda-matrix}
\mathcal{O}_{\lambda}=
\left(
\begin{array}{cccc}
0 & 0 & \cdots & 0 \\
0 & 0 & \cdots & 0 \\
\vdots & \vdots & \ddots & \vdots \\
\lambda d_k-b_k & 0 & \cdots & 0 \\
\end{array}
\right).
\end{align}
Pre-multiplying \eqref{eqn: eigenvalue relation suggested by data}
by $s_{[0,n]}^{L}$ gives the relation
\begin{align}
\label{eqn: pre-multiplying gives the relation in lemma}
\lefteqn{s_{[0,k]}^{L}\mathcal{H}_{[0,k]}p_{[0,k]}^{R}+
s_{[k+1,n]}^{L}\mathcal{H}_{[k+1,n]}p_{[k+1,n]}^{R}-
s_{[0,k]}^{L}\mathcal{O}_{\lambda}p_{[k+1,n]}^{R}-
s_{[k+1,n]}^{L}\mathcal{O}_{\lambda}^{\ast}p_{[0,k]}^{R}
}
\nonumber\\
&&=
\lambda(s_{[0,k]}^{L}\mathcal{J}_{[0,k]}p_{[0,k]}^{R}+
s_{[k+1,n]}^{L}\mathcal{J}_{[k+1,n]}p_{[k+1,n]}^{R}).
\end{align}
Similarly, from $s_{[0,n]}^{L}(\mu\mathcal{J}_{[0,n]}-\mathcal{H}_{[0,n]})=0$,
we obtain
\begin{align*}
\lefteqn{
s_{[0,k]}^{L}\mathcal{H}_{[0,k]}p_{[0,k]}^{R}+
s_{[k+1,n]}^{L}\mathcal{H}_{[k+1,n]}p_{[k+1,n]}^{R}-
s_{[0,k]}^{L}\mathcal{O}_{\mu}p_{[k+1,n]}^{R}-
s_{[k+1,n]}^{L}\mathcal{O}_{\mu}^{\ast}p_{[0,k]}^{R}
}
\\
&&=
\mu(s_{[0,k]}^{L}\mathcal{J}_{[0,k]}p_{[0,k]}^{R}+
s_{[k+1,n]}^{L}\mathcal{J}_{[k+1,n]}p_{[k+1,n]}^{R}),
\end{align*}
which used with
\eqref{eqn: pre-multiplying gives the relation in lemma}
to eliminate
$\mathcal{H}_{[0,n]}$ and $\mathcal{H}_{[k+1,n]}$
gives
\begin{align*}
\lefteqn{
s_{[0,k]}^{L}[\mathcal{O}_{\mu}-\mathcal{O}_{\lambda}]p_{[k+1,n]}^{R}+
s_{[k+1,n]}^{L}[\mathcal{O}_{\mu}^{\ast}-\mathcal{O}_{\lambda}^{\ast}]p_{[0,k]}^{R}
}
\\
&&=
(\lambda-\mu)(s_{[0,k]}^{L}\mathcal{J}_{[0,k]}p_{[0,k]}^{R}+
s_{[k+1,n]}^{L}\mathcal{J}_{[k+1,n]}p_{[k+1,n]}^{R}).
\end{align*}
The left hand side can be further simplified to finally obtain
\begin{align*}
s_{[0,k]}^{L}\mathcal{J}_{[0,k]}p_{[0,k]}^{R}+
s_{[k+1,n]}^{L}\mathcal{J}_{[k+1,n]}p_{[k+1,n]}^{R}+
d_{k}(s_{k+1}^{L}p_k^{R}+s_{k}^{L}p_{k+1}^{R})=0,
\end{align*}
which, as $\mu\rightarrow\lambda$,
implies
\begin{align}
\label{eqn: minus determinant expression}
d_k
\left|
  \begin{array}{cc}
    p_{k}^{L} & -p_{k}^{R} \\
    p_{k+1}^{L} & p_{k+1}^{L} \\
  \end{array}
\right|
=-
\left[
(p_{[0,k]}^{L})^{\ast}\mathcal{J}_{[0,k]}p_{[0,k]}^{R}+
(p_{[k+1,n]}^{L})^{\ast}\mathcal{J}_{[k+1,n]}p_{[k+1,n]}^{R}
\right].
\end{align}
The next lemma provides a crucial characterization of
the pole $\alpha_k$, that is, it should not be a real number
if the entry $b_k$ is to be determined uniquely.
\begin{lemma}
\label{lemma: determinant is non-zero}
Suppose $\alpha_k\notin\mathbb{R}$. Then
\begin{align}
\label{eqn: determinant is non-zero}
\Delta_k=
p_{k+1}^{L}p_k^{R}
\left|
  \begin{array}{cc}
    s_{k}^{L} & s_k^{R}\\
    s_{k+1}^{L} & s_{k+1}^{R} \\
  \end{array}
\right|
-
s_{k+1}^{L}s_k^{R}
\left|
  \begin{array}{cc}
    p_{k}^{L} & p_k^{R}\\
    p_{k+1}^{L} & p_{k+1}^{R} \\
  \end{array}
\right|\neq0,
\end{align}
if $\lambda,\mu\notin\sigma(\mathcal{H}_{[0,i]},\mathcal{J}_{[0,i]})$ for
$i=k-1,k$.
\end{lemma}
\begin{proof}
Using the forms as suggested by
\eqref{eqn: given form of rational functions},
we first note that
\begin{align*}
\left|
  \begin{array}{cc}
    p_k^{L} & p_k^{R} \\
    p_{k+1}^{L} & p_{k+1}^{R} \\
  \end{array}
\right|=
-\frac{\mathcal{T}_k(\lambda)\mathcal{T}_{k+1}(\lambda)}
{\eta_k^{\lambda}\eta_{k+1}^{\lambda}
\prod_{t=0}^{k}|\alpha_t-\lambda|^2}
2i\rm{Im}\alpha_k\neq0,
\end{align*}
by given assumptions and where $\rm{Im}\alpha_k$
is the imaginary part of $\alpha_k$.
Then,
\begin{align*}
\lefteqn{
\Delta_k=
\frac{\mathcal{T}_k(\lambda)\mathcal{T}_k(\mu)
\mathcal{T}_{k+1}(\lambda)\mathcal{T}_k(\mu)}
{\eta_k^{\lambda}\eta_{k}^{\mu}\eta_{k+1}^{\lambda}\eta_{k+1}^{\mu}
\prod_{t=0}^{k-1}|\alpha_t-\lambda|^2|\alpha_t-\mu|^2}
}
\\
&&\times
\left[
\frac{1}{|\alpha_k-\lambda|^2(\overline{\alpha_k}-\mu)}-
\frac{1}{|\alpha_k-\mu|^2(\overline{\alpha_k}-\lambda)}
\right]
2i\rm{Im}{\alpha_k},
\end{align*}
which simplifies further to give
\begin{align*}
\Delta_k=
\frac{\mathcal{T}_k(\lambda)\mathcal{T}_k(\mu)
\mathcal{T}_{k+1}(\lambda)\mathcal{T}_k(\mu)}
{
\eta_k^{\lambda}\eta_{k}^{\mu}\eta_{k+1}^{\lambda}\eta_{k+1}^{\mu}
\prod_{j=0}^{k}|\alpha_j-\lambda|^2|\alpha_j-\mu|^2}
2i(\lambda-\mu)\rm{Im}\alpha_k.
\end{align*}
Since $\alpha_k\notin\mathbb{R}$, we have that $\Delta_k$ is not zero.
\end{proof}
Now, with $\Delta_k$ a non-zero purely imaginary number
or equivalently,
$i\Delta_k$ a non-vanishing real number, we proceed to show that
$b_k$ can be determined uniquely.
\begin{theorem}
\label{theorem: theorem for solution of b-k}
Suppose that the given spectral point points
$\lambda,\mu\notin\sigma(\mathcal{H}_{[0,i]}, \mathcal{J}_{[0,i]})$
for $i=k-1,k$.
If $\alpha_k\notin\mathbb{R}$, then
\begin{align}
\label{eqn: b-k final expression after solution}
b_{k}=(\lambda+\mu)d_k+
\frac{d_k}{\Delta_{k}}
\left[
\mu s_{k+1}^{L}s_{k}^{R}
\left|
  \begin{array}{cc}
    p_{k}^{L} & p_k^{R} \\
    p_{k+1}^{L} & p_{k+1}^{R} \\
  \end{array}
\right|-
\lambda p_{k+1}^{L}p_{k}^{R}
\left|
  \begin{array}{cc}
    s_{k}^{L} & s_k^{R} \\
    s_{k+1}^{L} & s_{k+1}^{R} \\
  \end{array}
\right|
\right],
\end{align}
where $\Delta_k$ is given by
\eqref{eqn: determinant is non-zero}.
\end{theorem}
\begin{proof}
We have from 
\eqref{eqn: rational equation for components}
\begin{align}
\label{eqn: solution eqn for p-R}
(\lambda d_{m-1}-\bar{b}_{m-1})p_{m-1}^{R}+
(\lambda c_m-a_m)p_m^{R}+
(\lambda d_m-b_m)p_{m+1}^{R}
&=0,
\nonumber\\
(\lambda d_{n-1}-\bar{b}_{n-1})p_{n-1}^{R}+(\lambda c_n-a_n)p_n^{R}
&=0,
\end{align}
and the corresponding equations
for the components of $p_{[0,k]}^{L}$
\begin{align}
\label{eqn: solution eqn for p-L}
(\lambda d_{m-1}-b_{m-1})p_{m-1}^{L}+
(\lambda c_m-a_m)p_m^{L}+
(\lambda d_m-\bar{b}_m)p_{m+1}^{L}
&=0,
\nonumber\\
(\lambda d_{n-1}-b_{n-1})p_{n-1}^{L}+(\lambda c_n-a_n)p_n^{L}
&=0.
\end{align}
Eliminating $p_m^{R}$ and $p_m^{L}$ between the first equations of
\eqref{eqn: solution eqn for p-R} and \eqref{eqn: solution eqn for p-L}
we obtain
\begin{align*}
\lefteqn{
(p_{m-1}^{L}p_{m}^{R}b_{m-1}-
p_{m}^{L}p_{m+1}^{R}b_{m})+
(p_{m+1}^{L}p_{m}^{R}\bar{b}_{m}-
p_{m}^{L}p_{m-1}^{R}\bar{b}_{m-1})
}
\nonumber\\
&&=
\lambda d_{m-1}(p_{m-1}^{L}p_{m}^{R}-p_{m}^{L}p_{m-1}^{R})+
\lambda d_{m}(p_{m+1}^{L}p_{m}^{R}-p_{m}^{L}p_{m+1}^{R}),
\end{align*}
which on summing respective sides
from $m=k+1$ to $m=n-1$,
gives
\begin{align}
\label{eqn: solution eqn after summing}
\lefteqn{
(p_{k}^{L}p_{k+1}^{R}b_{k}-p_{n-1}^{L}p_{n}^{R}b_{n-1})+
(p_{n}^{L}p_{n-1}^{R}\bar{b}_{n-1}-p_{k+1}^{L}p_{k}^{R}\bar{b}_{k})
}
\nonumber\\
&&=
\lambda d_{k}(p_k^{L}p_{k+1}^{R}-p_{k+1}^{L}p_{k}^{R})
+
\lambda d_{n-1}(p_{n}^{L}p_{n-1}^{R}-p_{n-1}^{L}p_{n}^{R}).
\end{align}
From the last two relations of
\eqref{eqn: solution eqn for p-R} and
\eqref{eqn: solution eqn for p-L},
we get
\begin{align}
\label{eqn: solution equation from last two equations}
p_{n-1}^{L}p_{n}^{R}b_{n-1}-p_{n}^{L}p_{n-1}^{R}\bar{b}_{n-1}=
\lambda d_{n-1}(p_{n-1}^{L}p_{n}^{R}-p_{n}^{L}p_{n-1}^{R}),
\end{align}
which when added to
\eqref{eqn: solution eqn after summing}
yields
\begin{align}
\label{eqn: simultaneous system of eqn--1}
p_{k}^{L}p_{k+1}^{R}b_{k}-p_{k+1}^{L}p_{k}^{R}\bar{b}_{k}=
\lambda d_k(p_{k}^{L}p_{k+1}^{R}-p_{k+1}^{L}p_{k}^{R}).
\end{align}
A computation similar to the relations
\eqref{eqn: solution eqn for p-R},
\eqref{eqn: solution eqn for p-L}
and
\eqref{eqn: solution eqn after summing}
for the eigenpair $(\mu, s_{[o,n]}^{R})$ gives
\begin{align}
\label{eqn: simultaneous system of eqn--2}
s_{k}^{L}s_{k+1}^{R}b_{k}-s_{k+1}^{L}s_{k}^{R}\bar{b}_{k}=
\mu d_k(s_{k}^{L}s_{k+1}^{R}-s_{k+1}^{L}s_{k}^{R}).
\end{align}
We solve the system of equations
\eqref{eqn: simultaneous system of eqn--1}
and
\eqref{eqn: simultaneous system of eqn--2}
for $b_k$ and $\bar{b}_k$.
First, the determinant of the system is
\begin{align*}
\Delta_{k}
&=
p_{k+1}^{L}p_{k}^{R}s_{k}^{L}s_{k+1}^{R}-
p_{k}^{L}p_{k+1}^{R}s_{k+1}^{L}s_{k}^{R}
=
\left|
  \begin{array}{cc}
    s_{k}^{L}p_{k+1}^{L} & s_k^{R}p_{k+1}^{R}\\
    p_k^{L}s_{k+1}^{L} & p_{k}^{R}s_{k+1}^{R} \\
  \end{array}
\right|
\\
&=
p_{k+1}^{L}p_k^{R}
\left|
  \begin{array}{cc}
    s_{k}^{L} & s_k^{R}\\
    s_{k+1}^{L} & s_{k+1}^{R} \\
  \end{array}
\right|
-
s_{k+1}^{L}s_k^{R}
\left|
  \begin{array}{cc}
    p_{k}^{L} & p_k^{R}\\
    p_{k+1}^{L} & p_{k+1}^{R} \\
  \end{array}
\right|,
\end{align*}
which by Lemma~\ref{lemma: determinant is non-zero} is non-zero.
It is now a matter of computation to obtain
\begin{align*}
\lefteqn{
\Delta_{k}b_{k}=
\lambda d_kp_{k+1}^{L}p_{k}^{R}s_{k+1}^{L}s_{k}^{R}+
\mu d_ks_{k}^{L}s_{k+1}^{R}p_{k+1}^{L}p_{k}^{R}
}
\nonumber\\
&&-
\lambda d_kp_k^{L}p_{k+1}^{R}s_{k+1}^{L}s_{k}^{R}-
\mu d_ks_{k+1}^{L}s_{k}^{R}p_{k+1}^{L}p_{k}^{R},
\end{align*}
which can be further simplified to obtain
\begin{align*}
\Delta_k b_k=(\lambda+\mu)d_k\Delta_k+
d_k[\lambda p_{k+1}^{L}p_{k}^{R}(s_{k+1}^{L}s_k^{R}-s_{k}^{L}s_{k+1}^{R})-
\mu s_{k+1}^{L}s_{k}^{R}(p_{k+1}^{L}p_k^{R}-p_k^{L}p_{k+1}^{R})],
\end{align*}
leading to \eqref{eqn: b-k final expression after solution}
and specifying the entry $b_k$ uniquely.
\end{proof}
A similar computation for $\bar{b}_{k}$ gives
\begin{align*}
%\label{eqn: b-k-bar final expression after solution}
\bar{b}_{k}=(\lambda+\mu)d_k+
\frac{d_k}{\Delta_{k}}
\left[
\mu s_{k}^{L}s_{k+1}^{R}
\left|
  \begin{array}{cc}
    p_{k}^{L} & p_k^{R} \\
    p_{k+1}^{L} & p_{k+1}^{R} \\
  \end{array}
\right|-
\lambda p_{k}^{L}p_{k+1}^{R}
\left|
  \begin{array}{cc}
    s_{k}^{L} & s_k^{R} \\
    s_{k+1}^{L} & s_{k+1}^{R} \\
  \end{array}
\right|
\right].
\end{align*}
Theorem~\ref{theorem: theorem for solution of b-k}
finds the unique expression for $b_k$.
Summing \eqref{eqn: solution eqn after summing} from
$m=j$ to $m=n-1$ for each $j=k+1,\cdots,n-1$ yields an expression similar to
\eqref{eqn: b-k final expression after solution} for each $b_j$,
$j=k+1,\cdots,n-2$.
The entry $b_{n-1}$ is found from the system of equations consisting of
\eqref{eqn: solution equation from last two equations}
and the equivalent equation in $\mu$.
The assumptions are
$\alpha_j\notin\mathbb{R}$ and
$\lambda,\mu\notin\sigma(\mathcal{H}_{[0,j]},\mathcal{J}_{[0,j]})$,
$j=k+1,\cdots,n-1$.
Thus, with $b_j$, $j=k,k+1,\cdots,n-1$ determined, the $a_j's$
are found using
\eqref{eqn: rational equation for components}
as
\begin{align}
\label{eqn: entries a-j-in terms of b-j}
a_i=
\left\{
  \begin{array}{ll}
    \lambda c_i+
\frac{
(\lambda d_{i-1}-\bar{b}_{i-1})p_{i-1}^{R}(\lambda)+
(\lambda d_i-b_i)p_{i+1}^{R}(\lambda)}
{p_i^{R}(\lambda)}, & \hbox{$i=k+1,k+2\cdots n-1$;} \\
    \lambda c_n+\frac{(\lambda d_{n-1}-\bar{b}_{n-1})p_{n-1}^{R}(\lambda)}
{p_n^{R}(\lambda)}, & \hbox{$i=n$.}
  \end{array}
\right.
\end{align}
This completes the reconstruction of the matrix $\mathcal{H}_{[0,n]}$.
\begin{remark}
Since $\lambda$ and $\mu$ are zeros of $\mathcal{P}_{n+1}(z)$, the assumptions
for the determination of the matrix $\mathcal{H}_{[0,n]}$ requires that
$\mathcal{P}_j(z)$, $j=k-1,k,\cdots,n$, do not vanish at $\lambda$ and $\mu$.
However, we emphasize that the determination of each entry $b_j$ requires that
$\mathcal{P}_{j-1}(z)$ and $\mathcal{P}_j(z)$ do not share a common zero at
$\lambda$ and $\mu$.
This condition is often implicit, both in the direct and inverse problems, in the form of
the requirement that the zeros of
$\mathcal{P}_{j-1}(z)$ and $\mathcal{P}_j(z)$ or equivalently, the eigenvalues of the
corresponding pencil matrices satisfy a separation property known as interlacing.
\end{remark}
\begin{corollary}
\label{corollary: conditions for b-k to be imaginary}
 For $j=k,k+1,\cdots,n-1$, $b_j$ is purely imaginary and equals $ih_j$ if
\begin{align}
\begin{split}
\label{eqn: conditions for b-k to be imaginary}
\left|
  \begin{array}{cc}
    p_{j}^{L} & p_{j}^{R} \\
    p_{j+1}^{L} & p_{j+1}^{R} \\
  \end{array}
\right|^2
&=
i\frac{2\mu h_j\Delta_j}{\lambda(\lambda-\mu)d_j}
\frac{
\left|
  \begin{array}{cc}
    p_{j}^{L} & -p_{j}^{R} \\
    p_{j+1}^{L} & p_{j+1}^{R} \\
  \end{array}
\right|
}
{
\left|
  \begin{array}{cc}
    s_{j}^{L} & -s_{j}^{R} \\
    s_{j+1}^{L} & s_{j+1}^{R} \\
  \end{array}
\right|
},
\\
\left|
  \begin{array}{cc}
    s_{j}^{L} & s_{j}^{R} \\
    s_{j+1}^{L} & s_{j+1}^{R} \\
  \end{array}
\right|^2
&=
i\frac{2\lambda h_j\Delta_j}{\mu(\lambda-\mu)d_j}
\frac{
\left|
  \begin{array}{cc}
    s_{j}^{L} & -s_{j}^{R} \\
    s_{j+1}^{L} & s_{j+1}^{R} \\
  \end{array}
\right|
}
{
\left|
  \begin{array}{cc}
    p_{j}^{L} & -p_{j}^{R} \\
    p_{j+1}^{L} & p_{j+1}^{R} \\
  \end{array}
\right|.
}
\end{split}
\end{align}
\end{corollary}
\begin{proof}
First, let us find the real and imaginary parts of $b_j$.
Since $\alpha_j\notin\mathbb{R}$, we write $b_j=x_j+iy_j$ to obtain from
\eqref{eqn: simultaneous system of eqn--1}
and
\eqref{eqn: simultaneous system of eqn--2}
the system of equations
\begin{align*}
x_j+i\dfrac{p_{j}^{L}p_{j+1}^{R}+p_{j+1}^{L}p_{j}^{R}}
{p_{j}^{L}p_{j+1}^{R}-p_{j+1}^{L}p_{j}^{R}}y_j
=\lambda d_j;
\quad
x_j+i\dfrac{s_{j}^{L}s_{j+1}^{R}+s_{j+1}^{L}s_{j}^{R}}
{s_{j}^{L}s_{j+1}^{R}-s_{j+1}^{L}s_{j}^{R}}y_j
=\mu d_j,
\end{align*}
which can be solved to yield
\begin{align*}
x_j&=
\frac{d_j}{2\Delta_j}
\left[
\mu
\left|
  \begin{array}{cc}
    p_{j}^{L} & -p_{j}^{R} \\
    p_{j+1}^{L} & p_{j+1}^{R} \\
  \end{array}
\right|
\left|
  \begin{array}{cc}
    s_{j}^{L} & s_{j}^{R} \\
    s_{j+1}^{L} & s_{j+1}^{R} \\
  \end{array}
\right|
-
\lambda\left|
  \begin{array}{cc}
    s_{j}^{L} & -s_{j}^{R} \\
    s_{j+1}^{L} & s_{j+1}^{R} \\
  \end{array}
\right|
\left|
  \begin{array}{cc}
    p_{j}^{L} & p_{j}^{R} \\
    p_{j+1}^{L} & p_{j+1}^{R} \\
  \end{array}
\right|
\right],
\\
y_j&=
\frac{(\lambda-\mu)d_j}{2i\Delta_j}
\left|
  \begin{array}{cc}
    p_{j}^{L} & p_{j}^{R} \\
    p_{j+1}^{L} & p_{j+1}^{R} \\
  \end{array}
\right|
\left|
  \begin{array}{cc}
    s_{j}^{L} & s_{j}^{R} \\
    s_{j+1}^{L} & s_{j+1}^{R} \\
  \end{array}
\right|.
\end{align*}
The above (which can be noted to be in the form
$x_j=A_jX_j-B_jY_j$ and $y_j=C_jX_jY_j$)
solved further for $x_j=0$ and $y_j=h_j$, gives
the required relations
\eqref{eqn: conditions for b-k to be imaginary}.
\end{proof}
If $\alpha_j\notin\mathbb{R}$, the proof of 
Lemma~\ref{lemma: determinant is non-zero}
implies that $y_j\neq0$ and hence $b_j\notin\mathbb{R}$,
$j=k,\cdots,n-1$.
However, if we assume $x_j=0$,
that is if $\lambda$, $\mu$ satisfy
\begin{align}
\label{eqn: condition for b-k to be imaginary--lambda-mu}
\frac{\lambda}{\mu}=
\frac{\left|
  \begin{array}{cc}
    p_{j}^{L} & -p_{j}^{R} \\
    p_{j+1}^{L} & p_{j+1}^{R} \\
  \end{array}
\right|
\left|
  \begin{array}{cc}
    s_{j}^{L} & s_{j}^{R} \\
    s_{j+1}^{L} & s_{j+1}^{R} \\
  \end{array}
\right|
}
{\left|
  \begin{array}{cc}
    s_{j}^{L} & -s_{j}^{R} \\
    s_{j+1}^{L} & s_{j+1}^{R} \\
  \end{array}
\right|
\left|
  \begin{array}{cc}
    p_{j}^{L} & p_{j}^{R} \\
    p_{j+1}^{L} & p_{j+1}^{R} \\
  \end{array}
\right|
},
\quad
j=k,k+1,\cdots,n-1,
\end{align}
then $b_j=iy_j=i\epsilon_j d_j$, $\epsilon_j\neq0$,
that is $b_j$ is a scalar multiple of $id_j$.
\begin{remark}
The emphasis on $b_j$ being purely imaginary
arises from a particular form of the pencil
$(z\mathcal{J}_{[0,n]}-\mathcal{H}_{[0,n]})$,
where in fact $b_j=id_j$ and $c_j=1$, $j=0,1,\cdots,n$.
As mentioned in Section~\ref{sec: Introduction},
such pencils appear in analytic function theory
and a case has been made to call such pencils as
Wall pencils.

Further, as an inverse approach to these pencils, the relation
\eqref{eqn: condition for b-k to be imaginary--lambda-mu}
shows that the expression on the right hand side of
\eqref{eqn: condition for b-k to be imaginary--lambda-mu}
must be a constant and equal to the ratio of the given spectral points
for $b_j$ to be at least purely imaginary.
Hence, if we begin with $b_j=id_j$, $j=0,\cdots,k-1$ and $c_j=1$, 
$j=0,\cdots,k$,
appropriate conditions can be added to 
Corollary~\ref{corollary: conditions for b-k to be imaginary} 
so that $b_j$ is equal to $d_j$, $j=k,\cdots,n-1$ and we obtain a Wall pencil.
\end{remark}
The matrix $\mathcal{J}_{[0,n]}$ in Wall pencils is
positive-definite, while no such assumption has been made in
the present manuscript.
However, since the matrix $\mathcal{H}_{[0,n]}$
has been reconstructed,
let us have a look in this direction.
Suppose the assumptions of
Theorem~\ref{theorem: theorem for solution of b-k}
hold.
We put $\lambda-\mu=h$ in
\eqref{eqn: Wronskian-result for J-(k+1) before mu tends to lambda}
to get
\begin{align}
\label{eqn: positive definite for J-k}
s_{[0,k]}^{L}\mathcal{J}_{[0,k]}p_{[0,k]}^{R}=
\frac{(\bar{b}_k-\mu d_k)s_{k+1}^{L}p_k^{R}-
(b_k-\lambda d_k)s_{k}^{L}p_{k+1}^{R}
}
{h},
\end{align}
so that as $\lambda\rightarrow\mu$ or $h\rightarrow0$,
we have the left hand side as
$(s_{[0,k]}^{R})^{\ast}\mathcal{J}_{[0,k]}s_{[0,k]}^{R}$.
If $\mu\neq\alpha_j$,
$j=0,1,\cdots,k$, $s_{k+1}^{R}$ is finite at $\mu$
so that by Lemma~\ref{lemma: vector Phi-1},
$s_{[0,k]}^{R}$ is a vector with finite component.
By L'Hopital's rule,
\eqref{eqn: positive definite for J-k}
yields
\begin{align}
\label{eqn: positive definite identity for J-k}
(s_{[0,k]}^{R})^{\ast}\mathcal{J}_{[0,k]}s_{[0,k]}^{R}=
(\bar{b}_k-\mu d_k)s_{k+1}^{R}(s_k^{R})'-
(b_k-\mu d_k)s_k^{L}(s_{k+1}^{R})'+
d_ks_{k}^{L}s_{k+1}^{R},
\end{align}
so that if $\mathcal{J}_{[0,k]}$ is a positive-definite matrix,
the right hand side above is a positive quantity.

The essence of this observation is the following. For $\Delta_k\neq0$ to hold,
it is necessary that $\alpha_k\notin\mathbb{R}$. Since $\mu\in\mathbb{R}$,
$\mu\neq\alpha_k$ follows trivially.
Recalling the set $\{\alpha_0,\alpha_1,\cdots,\alpha_{k-1}\}$ is just a permutation of
$\{b_0/d_0, b_1/d_1, \cdots, b_{k-1}/d_{k-1}\}$, it follows that
$\mu\neq b_j/d_j$, $j=0,\cdots,k-1$.
Hence, for \eqref{eqn: positive definite identity for J-k} to hold
as an identity with finite quantities on both sides, no zero of
$\mathcal{P}_{n+1}(z)=\det{(z\mathcal{J}_{[0,n]}-\mathcal{H}_{[0,n]})}$
should coincide with $b_j/d_j$, $j=0,1,\cdots,k-1$.
\section{A view with $m$-functions}
\label{sec: view with m-functions}
In this section, we have a look at the reconstruction of the matrix
$\mathcal{H}_{[0,n]}$ through the concept of $m$-functions.
For the general theory of $m$-functions arising in the context of
orthogonal polynomials, we refer to
\cite{Simon-OPUC-part-I-2005}
and for that in the context of linear pencil, we refer to
\cite{Beckermann-D-Z-linear-pencil-JAT-2010}.
But because of the problem under consideration,
we will have use only of the representation
\eqref{eqn: m-functions for finite pencil}
for a point outside the spectrum of the pencil.

In the present case, in addition to $p_m^{R}(z)$ as defined in
\eqref{eqn: components of the eigenvector Phi},
we will use the rational functions
\begin{align*}
q_0^{R}(z)=0,
\quad
q_m^{R}(z)=\frac{\mathcal{Q}_m(z)}{\prod_{j=0}^{m-1}(zd_j-b_j)},
\quad
m=1,\cdots,n,
\end{align*}
where $\mathcal{Q}_m(z)$ satisfy
\eqref{eqn: recurrence relation satisfied by numerator polynomials}
with initial conditions
\eqref{eqn: initial conditions}.
A key role will be played by the following relation
\begin{align}
\label{eqn: Liouville-Ostrogradsy-formula}
\mathcal{P}_{m+1}(z)\mathcal{Q}_{m}(z)-
\mathcal{P}_{m}(z)\mathcal{Q}_{m+1}(z)=
\prod_{j=0}^{m-1}(zd_j-b_j)(zd_j-\bar{b}_j),
\end{align}
called the Liouville-Ostrogradsky formula and
which follows by induction from the recurrence relation
\eqref{eqn: recurrence relation satisfied by numerator polynomials}
along with the initial conditions \eqref{eqn: initial conditions}.
Using
\eqref{eqn: Liouville-Ostrogradsy-formula},
the matrix representation of the bounded operator
$(\omega\mathcal{J}-\mathcal{H})^{-1}$, $\omega\in\rho(\mathcal{H}, \mathcal{J})$
has been found in terms of $p_m^{R(L)}(z)$ and $q_m^{R(L)}(z)$
\cite[Theorem 2.3]{Beckermann-D-Z-linear-pencil-JAT-2010}.
The inverse of banded matrices has been studied, for instance, in
\cite{Stanica-inverse-banded-JCAM-2013}, but
we follow \cite{Beckermann-D-Z-linear-pencil-JAT-2010}
to obtain a finite version, that is the inverse of
the pencil $z\mathcal{J}_{[0,n]}-\mathcal{H}_{[0,n]}$.
\begin{lemma}
\label{lemma: inverse of finite pencil}
Let us denote $m_j:=\mathfrak{m}(\omega,j)-\mathfrak{m}(\omega,n+1)$.
Then the inverse $\mathcal{R}_{[0,n]}(\omega)$ of
$(\omega\mathcal{J}_{[0,n]}-\mathcal{H}_{[0,n]})^{-1}$,
$\omega\in\rho(\mathcal{H}_{[0,n]}, \mathcal{J}_{[0,n]})$
is given by
\begin{align*}
\mathcal{R}_{[0,n]}(\omega)=
\left(
  \begin{array}{cccccc}
    p_0^{L}m_0p_0^{R}  & p_1^{L}m_0p_0^{R} & p_2^{L}m_0p_0^{R} & \cdots 
    & p_{n-1}^{L}m_0p_0^{R} & p_n^{L}m_0p_0^{R} \\
    p_0^{L}m_0p_1^{R}  & p_1^{L}m_1p_1^{R} & p_2^{L}m_1p_1^{R} & \cdots 
    & p_{n-1}^{L}m_1p_1^{R} & p_n^{L}m_1p_1^{R} \\
    p_0^{L}m_0p_2^{R}  & p_1^{L}m_1p_2^{R} & p_2^{L}m_2p_2^{R} & \cdots 
    & p_{n-1}^{L}m_2p_2^{R} & p_n^{L}m_2p_2^{R} \\
    \vdots & \vdots & \vdots & \ddots & \vdots  & \vdots \\
    p_0^{L}m_0p_{n-1}^{R}  & p_1^{L}m_1p_{n-1}^{R} & p_2^{L}m_2p_{n-1}^{R} & \cdots 
    & p_{n-1}^{L}m_{n-1}p_{n-1}^{R} & p_n^{L}m_{n-1}p_{n-1}^{R} \\
    p_0^{L}m_0p_n^{R}  & p_1^{L}m_1p_n^{R} & p_2^{L}m_2p_n^{R} & \cdots 
    & p_{n-1}^{L}m_{n-1}p_{n}^{R} & p_n^{L}m_{n}p_n^{R} \\
  \end{array}
\right).
\end{align*}
\end{lemma}
\begin{proof}
Consider the $1\times(n+1)$ vector
$
p_{[0,j])}^{L}:
=\left(
\begin{array}{ccccccc}
p_0^{L} & p_1^{L} & \cdots &p_j^{L}  & 0 & \cdots & 0 \\
\end{array}
\right)
$
and similarly the vector $q_{[0,j]}^{L}$.
Using \eqref{eqn: rational equation for components}
for the left eigenvector we obtain
\begin{align*}
p_{[0,j]}^{L}(\omega\mathcal{J}_{(0,n)}-\mathcal{H}_{(0,n)})
&=
-(\omega d_j-\bar{b}_j)p_{j+1}^{L}\vec{e}_j^{T}+
(\omega d_j-b_j)p_j^{L}\vec{e}_{j+1}^{T},
\\
q_{[0,j]}^{L}(\omega\mathcal{J}_{(0,n)}-
\mathcal{H}_{(0,n)})
&=
\vec{e}_0^{T}-(\omega d_j-\bar{b}_j)q_{j+1}^{L}\vec{e}_j^{T}+
(\omega d_j-b_j)q_j^{L}\vec{e}_{j+1}^{T},
\end{align*}
which in view of
\eqref{eqn: Liouville-Ostrogradsy-formula}
leads to
\begin{align}
\label{eqn: in view of eqn after L-O formula}
[q_j^{R}p_{[0,j]}^{L}-p_j^{R}q_{[0,j]}^{L}]
(\omega\mathcal{J}_{[0,n]}-\mathcal{H}_{[0,n]})=
\vec{e}_j^{T}-p_j^{R}\vec{e}_0^{T}.
\end{align}
The vectors $p_{[0,n]}^{L}$ and $q_{[0,n]}^{L}$
with the above computation yield
\begin{align*}
[q_{n+1}^{R}p_{[0,n]}^{L}-p_{n+1}^{R}q_{(0,n)}^{L}]
(\omega\mathcal{J}_{[0,n]}-\mathcal{H}_{[0,n]})=
-p_{n+1}^{R}\vec{e}_0^{T},
\end{align*}
which in terms of $m$-functions can also be written as
\begin{align}
\label{eqn: in terms of m-function equation}
[q_{[0,n]}^{L}-\mathfrak{m}(\omega,n+1)p_{[0,n]}^{L}]
(\omega\mathcal{J}_{[0,n]}-\mathcal{H}_{[0,n]})=
\vec{e}_0^{T}.
\end{align}
Eliminating $\vec{e}_0^{T}$ between
\eqref{eqn: in view of eqn after L-O formula}
and \eqref{eqn: in terms of m-function equation},
we obtain
\begin{align*}
p_j^{R}[q_{[0,n]}^{L}-\mathfrak{m}(\omega, n+1)p_{[0,n]}^{L}-
q_{[0,j]}^{L}+\mathfrak{m}(\omega,j)p_{[0,j]}^{L}]
(\omega\mathcal{J}_{[0,n]}-\mathcal{H}_{[0,n]})
=\vec{e}_j^{T},
\end{align*}
which upon further simplification gives the matrix
$R_{[0,n]}(\omega)$.
\end{proof}
In compact form, the $(i,j)^{th}$ entry of $\mathcal{R}_{[0,n]}(\omega)$ is given by
$p_j^{L}m_{\min{(i,j)}}p_i^{R}$.
Next, for $\omega\in\rho(\mathcal{H}_{[0,n]},\mathcal{J}_{[0,n]})$,
let us factorize
\begin{align}
\label{eqn: LDU decomposition}
\mathcal{R}_{[0,n]}(\omega)=(\omega\mathcal{J}_{[0,n]}-\mathcal{H}_{[0,n]})^{-1}=
\mathcal{L}_{[0,n]}(\omega)\mathcal{D}_{[0,n]}(\omega)\mathcal{U}_{[0,n]}(\omega).
\end{align}
Then, owing to the Hermitian nature of the eigenvalue equation involved,
we choose
\begin{align*}
\mathcal{U}_{[0,n]}(\omega)=
\mathcal{L}_{[0,n]}^{\ast}(\omega),
\quad\mbox{where}\quad
\mathcal{L}_{[0,n]}(\omega)=
\left(
  \begin{array}{cccccc}
    p_0^{R} & 0 & 0 & \cdots & 0 & 0 \\
    p_1^{R} & p_1^{R} & 0 & \cdots & 0 & 0 \\
   \vdots & \vdots & \ddots & \ddots & \vdots & \vdots \\
    p_{n-1}^{R} & p_{n-1}^{R} & p_{n-1}^{R} & \cdots & p_{n-1}^R & 0 \\
    p_n^{R} & p_{n}^{R} & p_{n}^{R} & \cdots & p_n^{R} & p_n^{R} \\
  \end{array}
\right),
\end{align*}
and $\mathcal{D}_{[0,n]}(\omega)$ is the diagonal matrix
diag$(d_0(\omega), d_1(\omega), \cdots, d_n(\omega))$
where
\begin{align*}
d_0(\omega)=m_0,
\quad
d_j(\omega)=(m_{j}-m_{j-1})
=\mathfrak{m}(\omega, j)-\mathfrak{m}(\omega, j-1)
\quad j=1,2,\cdots,n.
\end{align*}
As a matter of verification, with $m_{-1}:=0$, we have in the right
hand side of \eqref{eqn: LDU decomposition}
\begin{align*}
i^{th}\hbox{ row} \times j^{th}\hbox{ column} =
p_i^{R}\sum_{k=0}^{\min{(i,j)}}(m_k-m_{k-1})p_{j}^{L}=
p_j^{L}m_{\min{(i,j)}}p_i^{R}.
\end{align*}
With this decomposition we can easily invert
$\mathcal{R}_{[0,n]}(\omega)$ again, so that
$[\mathcal{R}_{[0,n]}(\omega)]^{-1}=(\omega\mathcal{J}_{[0,n]}-\mathcal{H}_{[0,n]})$
will be a matrix in which the entries are given in terms of the $m$-functions.
We illustrate this for the trailing submatrix
$[\Psi_{[k+1,n]}(\omega)]^{-1}$.
\begin{lemma}
\label{lemma: inverse of trailing submatrix}
Suppose $\mathfrak{m}(\omega,k+1)\neq\mathfrak{m}(\omega,n+1)$ and
$\mathfrak{m}(\omega, i)\neq\mathfrak{m}(\omega, i-1)$, $i=k+2,\cdots,n$.
The entries of the inverse
of the trailing sub-matrix $\Psi_{[k+1,n]}(\omega)$
 are given by
\begin{align}
\label{eqn: entries of trailing matrix}
[\Psi_{[k+1,n]}(\omega)]^{-1}_{i,j}=
\left\{
  \begin{array}{ll}
    \frac{1}{p_{i}^{L}[\mathfrak{m}(\omega, i)-\mathfrak{m}(\omega, i-1)]p_{i}^{R}}+
    \frac{1}{p_{j}^{L}[\mathfrak{m}(\omega, j+1)-\mathfrak{m}(\omega, j)]p_{j}^{R}}, & \hbox{i=j;} \\
    \frac{-1}{p_{i}^{L}[\mathfrak{m}(\omega, j)-\mathfrak{m}(\omega, i)]p_{j}^{R}}, & \hbox{$i<j$,}\\
    0, & \hbox{$|i-j|>1$,}
  \end{array}
\right.
\end{align}
for $i,j=k+2,k+3,\cdots,n-1$, while
\begin{align}
\label{eqn: entries of trailing matrix first and last}
[\Psi_{[k+1,n]}(\omega)]^{-1}_{i,j}=
\left\{
  \begin{array}{ll}
    \frac{1}{p_{i}^{L}[\mathfrak{m}(\omega,i)-\mathfrak{m}(\omega, n+1)]p_{i}^{R}}+
    \frac{1}{p_{j}^{L}[\mathfrak{m}(\omega, j+1)-\mathfrak{m}(\omega, j)]p_{j}^{R}}, & \hbox{i=j=k+1;} \\
    \frac{1}{p_{i}^{L}[\mathfrak{m}(\omega,i)-\mathfrak{m}(\omega, i-1)]p_{i}^{R}}, & \hbox{i=j=n.}
  \end{array}
\right.
\end{align}
\end{lemma}
\begin{proof}
We start with the decomposition
$\mathrm{R}_{[k+1,n]}(\omega)=
\mathrm{L}_{[k+1,n]}(\omega)\mathrm{D}_{[k+1,n]}(\omega)\mathrm{U}_{[k+1,n]}(\omega)$
where
\begin{align*}
\mathrm{U}_{[k+1,n]}(\omega)=\mathrm{L}_{[k+1,n]}^{\ast}(\omega)
\quad\mbox{with}\quad
\mathrm{L}_{[k+1,n]}(\omega)=
\left(
  \begin{array}{cccccc}
    p_{k+1}^{R} & 0 & 0 & \cdots & 0 & 0 \\
    p_{k+2}^{R} & p_{k+2}^{R} & 0 & \cdots & 0 & 0 \\
   \vdots & \vdots & \ddots & \ddots & \vdots & \vdots \\
    p_{n-1}^{R} & p_{n-1}^{R} & p_{n-1}^{R} & \cdots & p_{n-1}^R & 0 \\
    p_n^{R} & p_{n}^{R} & p_{n}^{R} & \cdots & p_{n}^{R} & p_{n}^{R} \\
  \end{array}
\right)
\end{align*}
and $\mathrm{D}_{[k+1,n]}$=diag$\{d_{k+1}, d_{k+2}, \cdots, d_{n}\}$ given by
\begin{align*}
d_{k+1}(\omega)
&=m_{k+1}=\mathfrak{m}(\omega, k+1)-\mathfrak{m}(\omega,n+1),
\\
d_j(\omega)
&=(m_{j}-m_{j-1})
=\mathfrak{m}(\omega, j)-\mathfrak{m}(\omega, j-1),
\quad j=k+2,k+3,\cdots,n.
\end{align*}
It can be easily verified that
\begin{align*}
[\mathrm{L}_{[k+1,n]}(\omega)]^{-1}
=
\left(
  \begin{array}{cccccc}
    1/p_{k+1}^{R} & 0 & 0 & \cdots & 0 & 0 \\
    -1/p_{k+1}^{R} & 1/p_{k+2}^{R} & 0 & \cdots & 0 & 0 \\
%    0 & -1/p_{k+2}^{R} & 1/p_{k+3}^{R} & \cdots & 0 & 0 \\
    \vdots & \vdots & \ddots & \ddots & \vdots & \vdots \\
    0 & 0 & 0 & \cdots & 1/p_{n-1}^{R} & 0 \\
    0 & 0 & 0 & \cdots & -1/p_{n-1}^{R} & 1/p_{n}^{R} \\
  \end{array}
\right),
\end{align*}
and $[\mathrm{U}_{[k+1,n]}(\omega)]^{-1}=[\mathrm{L}_{[k+1,n]}^{\ast}(\omega)]^{-1}$.
Then,
\begin{align*}
[\mathrm{R}_{[k+1,n]}(\omega)]^{-1}=
[\mathrm{U}_{[k+1,n]}(\omega)]^{-1}[\mathrm{D}_{[k+1,n]}(\omega)]^{-1}
[\mathrm{L}_{[k+1,n]}(\omega)]^{-1},
\end{align*}
gives the required entries
\eqref{eqn: entries of trailing matrix}
and
\eqref{eqn: entries of trailing matrix first and last}.
\end{proof}
We are now ready to view the entries of
$\mathcal{H}_{[0,n]}$ in terms of $m$-functions.
Let us denote
\begin{align*}
\mathfrak{m}_{i}^{j}(\omega)=
\frac{1}{p_i^{L}[\mathfrak{m}(\omega, j)-\mathfrak{m}(\omega, i)]p_i^{R}}
\quad\mbox{and}\quad
\widehat{\mathfrak{m}}_{i}^{j}(\omega)=
\frac{1}{p_i^{L}[\mathfrak{m}(\omega, j)-\mathfrak{m}(\omega, i)]p_j^{R}}
\end{align*}
\begin{theorem}
\label{theorem: reconstruction via m-functions}
Suppose the $m$-functions $\mathfrak{m}(\omega,j)$ of the pencil
$(\omega\mathcal{J}_{[0,j]}-\mathcal{H}_{[0,j]})$ are known
and satisfy the assumptions of
Lemma~\ref{lemma: inverse of trailing submatrix}
for $j=k+1, k+2,\cdots,n,n+1$.
If $\omega\in\rho(\mathcal{H}_{0,k}, \mathcal{J}_{[0,k]})$,
then the matrix $\mathcal{H}_{[0,n]}$ can be reconstructed with
the entries given by
\begin{align}
\label{eqn: entries b-j-m-functions}
b_j=
\omega d_j+
%\frac{1}{p_{j}^{L}\mathfrak{m}(\omega,j+1)-\mathfrak{m}(\omega,j)p_{j+1}^{R}},
\widehat{\mathfrak{m}}_{j+1}^{j}(\omega),
\quad
j=k+1,k+2,\cdots,n-1,
\end{align}
and
\begin{align}
\label{eqn: entries a-j-m-functions}
a_j
=
\left\{
  \begin{array}{ll}
    \omega c_{j}-
\mathfrak{m}_{j}^{j+1}(\omega)+
\mathfrak{m}_{j}^{n+1}(\omega)+
\frac{|\omega d_k-b_k|^2}{\mathfrak{m}_{j-1}^{j}(\omega)}
& \hbox{$j=k+1$;}
\\
    \omega c_j+
\mathfrak{m}_{j}^{j-1}(\omega)-\mathfrak{m}_{j}^{j+1}(\omega),
& \hbox{$j=k+2,\cdots,n-1$,}
\\
    \omega c_j+
\mathfrak{m}_{j}^{j-1}(\omega),
& \hbox{$j=n$.}
 \end{array}
\right.
\end{align}
\end{theorem}
\begin{proof}
We use the following representation of the inverse obtained through
the concept of Schur's complement
\cite{Carlson-Schur-complement-LAA-1986}
\begin{align}
\label{eqn: Schur complement general formula}
\left(
  \begin{array}{cc}
    A & B \\
    C & D \\
  \end{array}
\right)^{-1}
=
\left(
  \begin{array}{cc}
    I & -A^{-1}B \\
    0 & I \\
  \end{array}
\right)
\left(
  \begin{array}{cc}
    A^{-1} & 0 \\
    0 & [D-CA^{-1}B]^{-1} \\
  \end{array}
\right)
\left(
  \begin{array}{cc}
    I & 0 \\
    -CA^{-1} & I \\
  \end{array}
\right),
\end{align}
where $I$ is the identity matrix of appropriate order.
Since $\Psi_{[k+1,n]}(\omega)$ is the trailing submatrix of the inverse
of the pencil, we also have
\begin{align}
\label{eqn: Schur complement our notations}
\left(
  \begin{array}{cc}
    w\mathcal{J}_{[0,k]}-\mathcal{H}_{[0,k]} & \mathcal{O}_{\omega} \\
    \mathcal{O}_{\omega}^{\ast} & w\mathcal{J}_{[k+1,n]}-\mathcal{H}_{[k+1,n]} \\
  \end{array}
\right)^{-1}
=
\left(
  \begin{array}{cc}
    \Psi_{[0,k]}(\omega) & \ast \\
    \ast & \Psi_{[k+1,n]}(\omega) \\
  \end{array}
\right).
\end{align}
Since $\omega\in\rho(\mathcal{H}_{[0,k]},\mathcal{J}_{[0,k]})$,
we can substitute $A=\omega\mathcal{J}_{[0,k]}-\mathcal{H}_{[0,k]}$,
$B=\mathcal{O}_{\omega}$ given by
\eqref{eqn: O-lambda-matrix}
and
$D=\omega\mathcal{J}_{[k+1,n]}-\mathcal{H}_{[k+1,n]}$.
Then comparing the respective blocks
of \eqref{eqn: Schur complement general formula}
and
\eqref{eqn: Schur complement our notations}, we get
\begin{align}
\label{eqn: comparison of respective blocks}
(\omega\mathcal{J}_{[k+1,n]}-\mathcal{H}_{[k+1,n]})
=
[\Psi_{[k+1,n]}(\omega)]^{-1}
+
\mathcal{O}_\omega^{\ast}
(\omega\mathcal{J}_{[0,k]}-\mathcal{H}_{[0,k]})^{-1}
\mathcal{O}_{\omega}.
\end{align}
We use Lemma~\ref{lemma: inverse of finite pencil}
to obtain the inverse $\mathcal{R}_{[0,k]}$ of
$\omega\mathcal{J}_{[0,k]}-\mathcal{H}_{[0,k]}$
so that
\begin{align*}
\mathcal{O}_{\omega}^{\ast}[\omega\mathcal{J}_{[0,k]}-\mathcal{H}_{[0,k]}]^{-1}
\mathcal{O}_{\omega}
=
p_{k}^{L}[\mathfrak{m}(\omega, k)-\mathfrak{m}(\omega, k+1)]p_k^{R}|\omega d_k-b_k|^2
\vec{e}_0\vec{e}_0^{T},
%\left(
%   \begin{array}{cccc}
%     p_{k}^{R}[\mathfrak{m}(\omega, k)-\mathfrak{m}(\omega, k+1)]p_k^{L}|\omega d_k-b_k|^2 & 0 & \cdots & 0 \\
%     0 & 0 & \cdots & 0 \\
%     \vdots & \vdots & \ddots & 0 \\
%     0 & 0 & 0 & 0 \\
%   \end{array}
% \right)
\end{align*}
where $\vec{e}_0\in\mathbb{R}^{k+1}$.
The entries of $[\Psi_{[k+1,n]}]^{-1}$
obtained from
\eqref{eqn: entries of trailing matrix}
and
\eqref{eqn: entries of trailing matrix first and last}
and used in
\eqref{eqn: comparison of respective blocks}
yields the required expressions
\eqref{eqn: entries a-j-m-functions}
and
\eqref{eqn: entries b-j-m-functions}.
%we get
%%
%\begin{align*}
%\omega\mathcal{J}_{[k+1,n]}-\mathcal{H}_{[k+1,n]}
%=
%[\mathrm{R}_{[k+1,n]}(\omega)]^{-1}+
%\mathcal{O}_{\omega}^{\ast}[\omega\mathcal{J}_{[0,k]}-\mathcal{H}_{[0,k]}]^{-1}
%\mathcal{O}_{\omega},
%\end{align*}
%%
%which upon comparison of respective entries yields
\end{proof}
It may be observed that
the rational functions $q_j^{R(L)}(\omega)$ are only intermediary
since the final expressions for the entries depend on the $m$-functions
defined by
\eqref{eqn: m-functions for finite pencil}.
As opposed to
\eqref{eqn: entries a-j-in terms of b-j}, the $a_j's$ in the present case,
except for $a_{k+1}$, are determined independent of $b_j's$
while for $a_{k+1}$, in addition to $b_k$, we need prior
information on $\mathfrak{m}(\omega, n+1)$.
Finally, the results of the present section can themselves be
seen as a sort of inverse problem of reconstructing the matrix
$\mathcal{H}_{[0,n]}$ from the knowledge of $m$-functions,
components of an eigenvector and a point in the resolvent set
of the linear pencil.

\begin{thebibliography}{99}
%
\bibitem{Aptekarev-resolvent-criteria-PAMS-1995}
A. I. Aptekarev, V. Kaliaguine\ and\ W. Van Assche,
Criterion for the resolvent set of nonsymmetric tridiagonal operators,
Proc. Amer. Math. Soc.
{\bf 123} (1995), no.~8, 2423--2430.
%
\bibitem{Baker-Morris-Pade-1981}
G. A. Baker, Jr.\ and\ P. Graves-Morris,
{\it Pad\'{e} approximants. Part I},
Encyclopedia of Mathematics and its Applications, 13,
Addison-Wesley Publishing Co., Reading, MA, 1981.
%
\bibitem{Beckermann-Weyl-function-Const.-Apprx-1997}
B. Beckermann, V. Kaliaguine,
The diagonal of the Padé table and the approximation of the Weyl function of second
order difference operators,
Constr. Approx.
13 (1997) 481–510.
%
\bibitem{Beckermann-D-Z-linear-pencil-JAT-2010}
B. Beckermann, M. Derevyagin\ and\ A. Zhedanov,
The linear pencil approach to rational interpolation,
J. Approx. Theory
{\bf 162} (2010), no.~6, 1322--1346.
%
\bibitem{Carlson-Schur-complement-LAA-1986}
D. Carlson,
What are Schur complements, anyway?,
Linear Algebra Appl.
{\bf 74} (1986), 257--275.
%
\bibitem{Chu-SIAM-review-1998}
M. T. Chu,
Inverse eigenvalue problems,
SIAM Rev.
{\bf 40} (1998), no.~1, 1--39.
%
\bibitem{Maxim-note-JAT-2017}
M. Derevyagin,
A note on Wall's modification of the Schur algorithm
and linear pencils of Jacobi matrices,
 J. Approx. Theory
{\bf 221} (2017), 1--21.
%
%\bibitem{Gladwell-matrix-inverse-chapter}
%G. M. L. Gladwell,
%Matrix inverse eigenvalue problems,
%in {\it Dynamical inverse problems: theory and application},
%1--28, CISM Courses and Lect., 529, SpringerWienNewYork, Vienna.
%
%\bibitem{Hald-IEVP-Jacobi-LAA-1976}
%O. H. Hald,
%Inverse eigenvalue problems for Jacobi matrices,
%Linear Algebra Appl.
%{\bf 14} (1976), no.~1, 63--85.
%
\bibitem{Ismail-Ranga-GEP-LAA-2019}
M. E. H. Ismail\ and\ A. Sri Ranga,
$R_{II}$ type recurrence, generalized eigenvalue problem
and orthogonal polynomials on the unit circle,
Linear Algebra Appl.
{\bf 562} (2019), 63--90.
%
\bibitem{Stanica-inverse-banded-JCAM-2013}
E. K\i l\i \c{c}\ and\ P. Stanica,
The inverse of banded matrices,
J. Comput. Appl. Math.
{\bf 237} (2013), no.~1, 126--135.
%
\bibitem{Lancaster-Ye-IGEP-LAA-1988}
P. Lancaster\ and\ Q. Ye,
Inverse spectral problems for linear and
quadratic matrix pencils,
Linear Algebra Appl.
{\bf 107} (1988), 293--309.
%
\bibitem{Sen-Sharma-LAA-2014}
M. Sen\ and\ D. Sharma,
Generalized inverse eigenvalue problem for matrices
whose graph is a path,
Linear Algebra Appl.
{\bf 446} (2014), 224--236.
%
\bibitem{Simon-OPUC-part-I-2005}
B. Simon, {\it Orthogonal polynomials on the unit circle. Part 1},
American Mathematical Society Colloquium Publications, 54, Part 1,
American Mathematical Society, Providence, RI, 2005.
%
\bibitem{Wall-book-1948}
H. S. Wall,
{\it Analytic Theory of Continued Fractions},
D. Van Nostrand Company, Inc.,
New York, NY, 1948.
%
\bibitem{Dai-Yuan-GIEP-JCAM-2009}
Y.-X. Yuan\ and\ H. Dai,
A generalized inverse eigenvalue problem in structural
dynamic model updating,
J. Comput. Appl. Math.
{\bf 226} (2009), no.~1, 42--49.
%
\bibitem{GIEP-Hamiltonian-AMC-2019}
H. Zhang\ and\ Y. Yuan, 
Generalized inverse eigenvalue problems for Hermitian 
and $J$-Hamiltonian/skew-Hamiltonian matrices, 
Appl. Math. Comput. 
{\bf 361} (2019), 609--616.
%
\bibitem{Zhedanov-GEP-JAT-1998}
A. Zhedanov,
Biorthogonal rational functions and the
generalized eigenvalue problem,
J. Approx. Theory
{\bf 101} (1999), no.~2, 303--329.

\end{thebibliography}
\end{document}